\documentclass{amsart}

\usepackage{amsthm,amssymb}
\usepackage[colorinlistoftodos]{todonotes}
\usepackage[margin=1.3in]{geometry}
\usepackage[utf8]{inputenc}

\title{Small Quotients of Braid Groups}
\author{Noah Caplinger \and Kevin Kordek }
\date{July 2020}

\newtheorem{theorem}{Theorem}
\newtheorem{lemma}{Lemma}

\newtheorem{question}{Question}
\newtheorem{definition}{Definition}

\newcommand{\inv}{^{-1}}
\newcommand{\s}{\sigma}
\newcommand{\Z}{\mathbb{Z}}

\begin{document}

\maketitle

\begin{abstract}
    We prove that the symmetric group $S_n$ is the smallest non-cyclic quotient of the braid group $B_n$ for $n=5,6$ and that the alternating group $A_n$ is the smallest non-trivial quotient of the commutator subgroup $B_n'$ for $n = 5,6,7,8$. We also give an improved lower bound on the order of any non-cyclic quotient of $B_n$. 
    
\end{abstract}

\section{Introduction}

Let $B_n$ denote the braid group on $n$ strands, and $B_n'$ its commutator subgroup. We are be interested in studying quotients of $B_n$. In light of the fact that the abelianization $B_n/B_n'$ is infinite cyclic, it is simple to create `uninteresting' homomorphisms $B_n \to G$ that factor through $\mathbb{Z}$. Such a homomorphism is said to be cyclic. Any non-cyclic map $B_n \to G$ gives a non-trivial map $B_n' \to G$ by restriction, so non-cyclic quotients of $B_n$ are closely related to non-trivial quotients of $B_n'$. 

There is a canonical map $B_n \to S_n$, which gives the alternating group $A_n$ as a quotient of $B_n'$ by restriction. It is not known whether there are quotients of smaller order. In \cite{finite_quotients}, Margalit asked whether $S_n$ and $A_n$ are indeed the smallest (resp. non-cyclic, non-trivial) quotients of $B_n$ and $B_n'$ respectively. We further break up this question as follows:

\begin{question}[Numerical level]
\label{S_n_numerical_level}

For $n \geq 5$, if $G$ is a non-cyclic quotient of $B_n$, must $|G| \geq n! \;$?
\end{question}

\begin{question}[Group level]
\label{S_n_group_level}

Is the following true for $n \geq 5$: if $G$ is a non-cyclic quotient of $B_n$ then either $G \cong S_n$ or $|G| > n!\;$ ? 
\end{question}

We could ask similar questions about $B_n'$ and $A_n$, but since $B_n'$ is perfect, its smallest non-trivial quotient must be non-abelian and simple. There are few non-isomorphic simple groups with the same order, so we will not distinguish between these questions for $B_n'$.

\begin{question}
\label{A_n_question}

Is the following true for $n \geq 5$: if $G$ is a non-trivial quotient of $B_n'$ then $G \cong A_n$ or $|G| > n!/2 \;$ ?

\end{question}

Our main result gives an affirmative answer to question \ref{A_n_question} for $n \in \{5,6,7,8\}$, to question \ref{S_n_numerical_level} for $n \in \{5,6,7\}$ and to question \ref{S_n_group_level} for $n \in \{5,6\}$. Thus, $A_n$ is the smallest quotient of $B_n'$ for $n \in \{5,6,7,8\}$, and $S_n$ is the smallest quotient of $B_n$  at the ``group level" (question \ref{S_n_group_level}) for $n \in \{5,6\}$ and at the ``numerical level" (question \ref{S_n_numerical_level}) for $n=7$.

\begin{theorem}
\label{A_n_group}
Let $n \in \{5,6,7,8\}$, and $H$ be a non-trivial quotient of $B_n'$. Then then $H \cong A_n$ or $|H| > n!/2$.
\end{theorem}

\begin{theorem}
\label{B_n_numerical}
Let $n = 7$, and $G$ be a non-cyclic quotient of $B_n$. Then $|G| \geq n!$.
\end{theorem}

\begin{theorem}
\label{B_n_group}
Let $n \in \{5,6\}$, and $G$ be a non-cyclic quotient of $B_n$. Then $G \cong S_n$ or $|G| > n!$.
\end{theorem}

We also give a refinement of the main theorem in \cite{finite_quotients}, which states that any non-cyclic quotient of $B_n$ must have order at least $2^{\lfloor n/2 \rfloor - 1}\lfloor n/2 \rfloor !$.

\begin{theorem}
\label{thmBound}
Let $n \geq 5$ and $G$ be a non-cyclic quotient of $B_n$. Then $$|G| \geq 3^{\lfloor n/2 \rfloor - 1}\lfloor n/2 \rfloor !$$
\end{theorem}

\vspace{.5cm}
\noindent
\textit{Prior results.} Artin \cite{Artin} classified all homomorphisms $B_n \to S_n$ with transitive image---for $n \neq 4,6$, they must be either cyclic or conjugate to the canonical map. Artin also described all exceptional homomorphisms in the case where $n \in \{4,6\}$. These results were greatly expanded by Lin \cite{Lin}, who broadly characterized  homomorphisms from braid groups to symmetric groups. Among many other results, he proved that any homomorphism $B_n \to S_k$ with $k < n$ is cyclic, and that all transitive homomorphisms $B_n \to S_m$ with $6 < n < m < 2n$ are cyclic. The general flavor of these results is that, with exceptions, the standard projection is essentially the only non-cyclic homomorphism from the braid group to a symmetric group of small enough order.

Zimmermann \cite{zimmermann2012} proved that for $g \in \{3,4\}$, the smallest minimal non-trivial quotient of the genus $g$ mapping class group $\text{Mod}(S_g)$ is the symplectic group $\text{Sp}_g(\Z/2)$. He then conjectured that the result held for $g \geq 3$, which was later proven by Kielak-Pierro \cite{KielakPierro}. Margalit's question can be seen as the braid group analogue of Zimmermann's.

In \cite{first_TSS} Margalit and the second author classified all homomorphisms from $B_n' \to B_n$ for $n \geq 7$. In doing so, they introduced the notion of a totally symmetric set, which will be the main tool in this paper. Chudnovsky, Li, Partin and the second author \cite{finite_quotients} then used totally symmetric sets to derive a lower bound on the order of a non-cyclic quotient, which we refine with theorem \ref{thmBound}.

\vspace{.3cm}

\textbf{Overview.} In section 2, we review the definition and properties of totally symmetric sets, which will be central to the paper. In section 3 we give proofs of theorems  \ref{A_n_group}, \ref{B_n_numerical}, and \ref{B_n_group}. In section 4, we give a proof of theorem \ref{thmBound}. In the Appendix we give an explanation of the algorithms used in the computer checks and prove lemma \ref{TSS_in_linear}.   

\vspace{.3cm}

\textbf{Acknowledgments.} The authors would like to thank Dan Margalit for the helpful comments and suggestions and for his constant encouragement throughout the project. We would also like to thank Caleb Partin for helpful conversations.

\section{Totally Symmetric Sets}
\label{TSS_section}

The main tool in our study of braid quotients is the theory of totally symmetric sets, introduced by Margalit and the second author in \cite{first_TSS}. A totally symmetric set $S$ is a set of pairwise commuting elements of a group $G$ such that every permutation of $S$ can be realized by conjugating $S$ by an element of $G$. 

\begin{definition}
Let $G$ be a group. We say $S  = \{g_1,\ldots, g_k\} \subset G$ is totally symmetric if 

\begin{enumerate}
    \item $g_ig_j = g_jg_i$ for all $i,j \in [k]$
    \item For every $\sigma \in S_k$, there exists some $h \in G$ so that $hg_ih\inv = g_{\sigma(i)}$ (the ``total symmetry" condition)
\end{enumerate}
\end{definition}

It will prove useful to think of totally symmetric sets in terms of group actions. If we let $G$ act on its subsets by conjugation, we say a subset $T \subset G$ is totally symmetric if its elements pairwise commute and the natural map $\text{Stab}(T) \to \text{Sym}(T) \cong S_{|T|}$ is surjective. Note also that $T$ is totally symmetric in both $G$ and $\text{Stab}(T)$.

\vspace{.5cm}

\noindent
\textit{Examples of totally symmetric sets.} The two main totally symmetric sets used in this paper are 

$$X_n = \{\s_{2i-1}\}_{i=1}^{i=\lfloor n/2 \rfloor } \subset B_n  \quad \text{and}$$  

$$X_n' =  \{\s_{2i-1} \s_1\inv \}_{i=2}^{\lfloor n/2 \rfloor} \subset B_n'.$$

\noindent
It is clear that the elements of $X_n$ and $X_n'$ pairwise commute, and an application of the change-of-coordinates principle from mapping class group theory (see \cite{Primer}) shows that they are totally symmetric in $B_n$. For a proof that $X_n'$ is also totally symmetric in $B_n'$, see \cite{finite_quotients}.

\vspace{.5cm}

The image of $X_n$ under the standard map $B_n \to S_n$ is $\{(1 \: 2), (3 \: 4), \ldots\}$, which also happens to be totally symmetric. This is no accident:

\begin{lemma}[Kordek--Margalit, \cite{first_TSS}]
If $S \subset G$ is a totally symmetric set and $f:G \to H$ is any homomorphism, then $f(S)$ is a totally symmetric set of cardinality $1$ or $|S|$. 
\end{lemma}

\begin{proof}
It is clear that $f(S)$ is totally symmetric in $H$. Let $S = \{g_1,\ldots, g_k\}$ and assume $f(g_1) = f(g_2)$. By total symmetry, for every $i \geq 3$ there is some $h \in G$ so that $hg_1h\inv = g_1$ and $hg_2h\inv = g_i\inv$. Then conjugating $f(g_1) = f(g_2)$ by $f(h)$ gives $f(g_1) = f(g_i)$ for all $i$. 
\end{proof}

This lemma gives us an indirect way to probe the injectivity of a map $G \to H$. Specifically, if $G$ has a totally symmetric set of cardinality $n$ and $H$ does not, then any map $G\to H$ must collapse that totally symmetric set to a single element. 

In the case of braid groups, we can do even better. If $H$ is some group without a totally symmetric set of cardinality 3, any map $f:B_6 \to H$ must send $X_6 = \{ \s_1, \s_3, \s_5 \}$ to a single element, meaning $\s_1\s_3\inv \in \text{ker}(f)$. Because $\s_1\s_3\inv$ normally generates $B_n'$ (for $n \geq 5$), $f$ must factor through the abelianization $B_6/B_6'$ and is therefore cyclic. Thus any non-cyclic quotient of $B_6$ must contain a totally symmetric set of cardinality 3. Variants of this fact will be used several times throughout this paper.

In light of this, it is useful be able to disprove the existence of totally symmetric sets of a certain cardinality in some group. We have two results of this type. The first is due to Chen, Kordek, and Margalit but first appeared in \cite{finite_quotients}.

\begin{lemma}[Chen--Kordek--Margalit]
\label{TSS_criterion}
Let $n \geq 1$ and suppose that $S$ is a totally symmetric subset of a group $G$ with $|S| = n$. If the elements of $S$ have finite order, then $|G| \geq 2^{n-1}n!\:$.

\end{lemma}

A full proof of this lemma can be found in \cite{finite_quotients}, and will be partially discussed in section \ref{sectionOnBoundTheorem}.

\begin{lemma}
\label{TSS_in_linear}
If $\mathbb{K}$ is a field and $S \subset GL_2(\mathbb{K})$ is a totally symmetric set, then $|S| \leq 2$
\end{lemma}

The proof of this lemma is long and computational, so we present it in the Appendix.

It is not known if there are any totally symmetric sets of cardinality $n$ in $GL_{n-1}(\mathbb{K})$, nor is it known whether such a set cannot exist. This knowledge would be useful in classifying homomorphisms from braid groups.

\section{Proof of Theorems \ref{A_n_group}, \ref{B_n_numerical}, and \ref{B_n_group}}

\noindent
For each $n \in \{5,6,7,8\}$, A basic outline for the proof of theorems \ref{A_n_group}, \ref{B_n_numerical}, and \ref{B_n_group} is as follows:

\begin{enumerate}
    \item Use the classification of finite simple groups to show that the smallest non-trivial quotient of $B_n'$ (which must be simple) is $A_n$.
    \item Show that the smallest non-cyclic quotient of $B_n$ has order $n!$.
    \item Show that $S_n$ is the only non-cyclic quotient of $B_n$ with order $n!$.
\end{enumerate}

\vspace{.2cm}

Due to our reliance on the classification of finite simple groups (for a full statement, see \cite{CFSG}), we will give a list of all simple non-abelian groups with order at most $8!/2$. 
\vspace{.2cm}

\begin{center}
\setlength{\tabcolsep}{10pt} 
\renewcommand{\arraystretch}{1.5}
\begin{tabular}{|c c | c c | c c|}

\hline
     60 & $A_5$                & 2520 & $A_7$       & 7920 & $M_{11}$ \\
     168 & PSL(2,7)           & 3420 & PSL(2,19) & 9828 & PSL(2,27) \\
     360 & $A_6 \cong \text{PSL}(2,9)$ & 4040 & PSL(2,16) & 12180 & PSL(2,29) \\
     504 & PSL(2,8)           & 5616 & PSL(3,3)  & 14880 & PSL(2,31) \\
     660 & PSL(2,11)          & 6048 & $G_2(2)'$   & 20160 & $A_8$ \\
     1096 & PSL(2,13)         & 6072 & PSL(2,23) & 20160 & PSL(3,4) \\
     2448 & PSL(2,17)         & 7800 & PSL(2,25) &&\\
    \hline
\end{tabular}
\end{center}

\vspace{.3cm}

\noindent
The $n=5$ case requires no additional machinery or computer assistance, nor does it use totally symmetric sets.

\vspace{.3cm}

\noindent
\textbf{Proof of theorems \ref{A_n_group}, \ref{B_n_numerical}, and \ref{B_n_group} for $n=5$}.

\begin{proof}

Being the smallest simple non-abelian group, $A_5$ is the smallest non-trivial quotient of $B_5'$.

It remains to prove $S_5$ is the smallest non-cyclic quotient of $B_5$. Let $f: B_5 \to G$ be a non-cyclic surjection. Because $f$ is non-cyclic, $f(B_5')$ is non-trivial. If $f(B_5')$ were not simple it would have a simple non-abelian quotient ($B_5'$ is perfect). Then, $f(B_5')$ must have order at least twice that of the smallest simple non-abelian group, ie at least $120 = 5!$, so $f(B_5)$ must also have order at least $120$. 

If $f(B_5')$ is simple, it can either have order larger than $120$, or be $A_5$. In the latter case, it must be of index at least 2, lest we have a map $B_5 \to S_5$ not conjugate to the canonical map. This would contradict a result by Artin \cite{Artin} which says for $n \not \in \{4,6\}$) any homomorphism $B_n \to S_n$ with transitive image must be cyclic or conjugate to the standard projection. Then $Im(f)$ contains $A_5$ as a subgroup of index at least 2, and therefore has order at least $120$. 

Because $B_5'$ is perfect, any non-cyclic quotient of $B_5$ must be non-solvable, so we will turn our attention to non-solvable groups of order 120. There are 3 such groups: 1. $S_5$, 2. $SL(2,5)$ and 3. $A_5 \times \mathbb{Z}_2$. Both $SL(2,5)$ and $A_5 \times \mathbb{Z}_2$ have further quotients to $A_5$, which would give a map $B_5 \to A_5$. This contradicts the same result by Artin. Thus, $S_5$ is the only quotient of $B_5$ with order $120$. 

\end{proof}

Given our reliance on the classification of finite simple groups, it will be useful to show certain homomorphisms from braid groups to simple groups do not exist. This will reduce the number of computer checks required.  

\begin{lemma}
\label{big_PSL}
Let $p \neq 2,3$ be prime, $r \in \mathbb{N}$. For $n \geq 6$, and $q = p^r$ there are no non-cyclic maps $B_n \to PSL(2, q)$. For $n \geq 8$, there are no non-trivial maps $B_n' \to PSL(2,q)$.
\end{lemma}

\begin{proof}
Both $X_6 \subset B_n$ and $X_8' \subset B_n'$ are totally symmetric sets of cardinality 3 for $n \geq 6$ and $n \geq 8$ respectively. As discussed in section \ref{TSS_section}, if $X_n$ or $X_n'$ are collapsed to a single element, the homomorphism in question must be cyclic or trivial respectively. To disprove the existence of non-cyclic and non-trivial maps to $PSL(2,q)$, we will show that $PSL(2,q)$ has no suitable totally symmetric sets of cardinality 3 which can be the images of $X_6$ or $X_8'$.

If such a totally symmetric set $S \subset PSL(2,q)$ existed, its stabilizer would be a subgroup of $PSL(2,q)$ which surjects onto $S_3$. The subgroups of $PSL(2,q)$ are listed below \cite{PSL_subgroups}.

\begin{enumerate}
    \item Elementary-abelian $p$-groups
    \item Cyclic groups of order $z$, where $z$ divides $(q \pm 1)/k$ with $k = \text{gcd}(q-1,2)$
    \item Dihedral subgroups of order $2z$ with $z$ as above
    \item Alternating groups $A_4$ if\footnote{listed for completeness, $p \neq 2$ in our case.} $p > 2$ or $p=2$ and $r \equiv 0 \text{ (mod 2)}$
    \item Symmetric groups $S_4$ if $q^2 - 1 \equiv 0 \text{ (mod 16)}$
    \item Alternating groups $A_5$ if $p=5$ or $q^2 - 1 \equiv 0 \text{ (mod 5)}$
    \item Semidirect products $C_p^m \rtimes C_t$ of elementary-abelian groups of order $p^m$ with cyclic groups of order $t$, where $t$ divides $p^m-1$ as well as $(q-1)/k$
    \item Groups $PSL(2,p^m)$ if $m \mid r$ and $PGL(2,p^m)$ if $2m \mid r$.
\end{enumerate}

\noindent
We will show that no subgroups listed can be the stabilizer of the image of $X_6$ or $X_8'$

It is simple to show that there are no totally symmetric sets of cardinality 3 in an abelian or dihedral group. The alternating group $A_4$ has order $12 < 2^{3-1} \cdot 3!$, and therefore cannot contain a totally symmetric set by lemma \ref{TSS_criterion}. The groups $A_5$ and $PSL(2,p^m)$ are both simple and therefore can not map to $S_3$ non-trivially, while $PGL(2,p^m)$ contains $PSL(2,p^m)$ as a subgroup of index 2, so it cannot surject onto $S_3$. This deals with cases 1, 2, 3, 4, 6, and 8. 

For case 7, note that every element of $C_p^m \triangleleft (C_p^m \rtimes C_t$) has prime order not dividing $6$, so any homomorphism $C_p^m \rtimes C_t$ factors through $C_t$ and therefore cannot be surjective.

For case 5, we first find all totally symmetric sets $T \subset S_4$ of cardinality 3 for which $S_4$ is the stabilizer. We have the natural (surjective) map $\text{Stab}(T) \cong S_4 \to S_3$ which must have kernel $\{e, (1 \: 2)(3 \: 4), (1\:3)(2 \: 4), (1 \: 4)(2 \: 3)\}$. The elements of $T$ pairwise commute, and therefore must be in this kernel (elements of $T$ act trivially on $T$ by conjugation). Finally, the elements of $T$ are in the same conjugacy class, so $T = \{(1 \: 2)(3 \: 4), (1\:3)(2 \: 4), (1 \: 4)(2 \: 3)\}$.

However, $T$ cannot be the image of $X_6$ or $X_8'$, as it is incompatible with the braid relation. Let $c_i$ be the image of $\s_i$ in the case of $B_n \to PSL(2,q)$ and the image of $\s_{i+2} \s_1\inv$ in the case of $B_n' \to PSL(2,q)$. Assume $\{c_1,c_3,c_5\} = T$. Then $c_ic_j = c_k$ for any distinct $i,j,k \in \{1,3,5\}$ and  

\begin{equation}
\begin{split}
    c_2 (c_3c_5) c_2 &= c_2 c_1 c_2 \\
    &= c_1 c_2 c_1 \\
    &= (c_3c_5) c_2 (c_3c_5) \\ 
    &= c_3 c_2 c_3 (c_5 c_5) \\ 
    &= c_2 c_3 c_2 (c_5 c_5)
\end{split}
\end{equation}

\noindent
so $c_2 c_3 c_2 c_5 = c_2 c_3 c_2 c_5 c_5$, meaning $c_5 = 1$, which is impossible.

\end{proof}

The $p = 2,3$ difficulty arises when considering $C_p^m \rtimes C_t$ as a possible stabilizer. This can be worked around in particular cases:

\begin{lemma}
\label{small_PSL}
Let $G = PSL(2,8)$ or $PSL(2,27)$. For $n \geq 6$, there are no non-cyclic maps $B_n \to G$, and for $n \geq 8$, there are no non-trivial maps $B_n' \to G$. 
\end{lemma}

\begin{proof}
If $G = PSL(2,8)$, the semi-direct product $C_p^m \rtimes C_t$ must have $t$ dividing $7$ (by the criteria in case 7). Then $|C_p^m \rtimes C_t|$ is not a multiple of $6$, and therefore cannot surject to $S_3$. Similarly, for $G = PSL(2,27)$, we know $t \mid 13$. 
\end{proof}

The only group of the form $PSL(2,q)$ we consider which is not covered by the two preceding lemmas is $PSL(2,16)$, which we will check with a computer. That we require $n\geq 8$ for homomorphisms $B_n' \to PSL(2,q)$ is a notable shortcoming which will also be patched with the help of computers. Fortunately, we do not need to deal with these cases after proving theorem \ref{A_n_group}.

\vspace{.3cm}

\noindent
\textbf{Proof of theorems \ref{A_n_group}, \ref{B_n_numerical}, and \ref{B_n_group} for $n=6$}.

We remind the reader of the proof strategy for theorems \ref{A_n_group}, \ref{B_n_numerical}, and \ref{B_n_group}, namely to use the classification of finite simple groups to find the smallest non-trivial quotient of $B_n'$, then bound the order of a non-cyclic quotient of $B_n$ by $n!$, and finally show that the symmetric group is the only non-cyclic quotient of order $n!$.

\begin{proof}
We require computer assistance to disprove the existence of three homomorphisms: 1. a non-trivial map $B_6' \to PSL(2,7)$ 2. a non-cyclic map $B_6 \to PGL(2,9)$ and 3. a non-cyclic map $B_6 \to M_{10}$ (the Mathieu group of degree 10). We will proceed assuming these homomorphisms do not exist. 

There are only two simple non-abelian groups of order smaller than $|A_6|$, namely $PSL(2,7)$ and $A_5$. The first has been ruled out by a computer check, and the second would give a non-trivial map $B_6' \to S_5$, contradicting a result by Lin \cite{Lin}. Then $A_6$ is the smallest non-trivial quotient of $B_6'$.

Let $f: B_6 \to G$ be non-cyclic. If $f(B_6')$ is not simple, it has a simple non-abelian quotient, which we have just seen cannot be $PSL(2,7)$ or $A_5$. Then $f(B_6')$ has a proper quotient of order $\geq 360$, so $|G| \geq 720$. If $f(B_6')$ is simple and has order less than  $720$, it must be one of $PSL(2,8)$, $PSL(2,11)$ or $A_6$. Both $PSL(2,8)$ and $PSL(2,11)$ cannot be the images of a map from $B_6$ (lemmas \ref{big_PSL}, \ref{small_PSL}), and therefore have index at least $2$ in $G$, meaning $|G| \geq 720$. Similarly, a non-cyclic map $B_6 \to A_6$ is impossible by Artin\footnote{We are careful to note that $n=6$ an exceptional case in this result. Artin described all exceptional homomorphisms, none of which include $A_6$ as an image. } \cite{Artin}. Thus, $|G| \geq 720$.

As before, we note that any non-cyclic quotient of $B_6$ must be non-solvable. There are 5 such groups of order $720$: 1. $S_6$, 2. $A_6 \times \Z_2$, 3. $SL(2,9)$, 4. $PGL(2,9)$ and 5. $M_{10}$. We rule out 4. and 5. with the help of computers, while 2. has a further quotient to $A_6$, which we have already discussed to be impossible. We may disqualify 3. since linear groups cannot have totally symmetric sets of cardinality 3 (lemma \ref{TSS_in_linear}).  Thus, $S_6$ is the smallest quotient of $B_6$.

\end{proof}

\noindent
\textbf{Proof of theorems \ref{A_n_group}, \ref{B_n_numerical}, and \ref{B_n_group} for $n=7$ and $n=8$}.

\begin{proof}
Seven is the largest value of $n$ not covered by lemma \ref{big_PSL}, and as such will require more extensive computer assistance. We check that $PSL(2,8)$, $PSL(2,11)$, $PSL(2,13)$, $PSL(2,17)$ and $PSL(2,16)$ are not the co-domain of any non-trivial map from $B_7'$.

From before, every map $B_6' \to PSL(2,7)$ is trivial, and therefore collapses $X_6'$ to a single element. Then every map $B_7' \to PSL(2,7)$ also collapses $X_6'$, and is therefore trivial. This deals with the $PSL(2,7)$ case. A result from Lin \cite{Lin} tells us there can be no non-trivial maps $B_7' \to A_k$ for $k=5,6$. Then $B_7'$ cannot map non-trivially to any simple group of order less than $7!/2$, and therefore $A_7$ is the smallest non-trivial quotient of $B_7'$.

Let $f:B_7 \to G$ be a non-cyclic surjection. Since $B_7'$ cannot map non-trivially to any simple group of order less than $7!/2$, the only way for $G$ to have order less than $7!$ is if $f(B_7') = G$. We may use lemma \ref{big_PSL} to conclude that if $f(B_7') = PSL(2,19)$, we must have $[G:f(B_7')] \geq 2$ and therefore $|G| > 7!$, while we know $f(B_7') = PSL(2,16)$ is impossible from our computer check. If $f(B_n')$ is not simple, we may use index arguments similar to those used in the $n=5,6$ cases. Then, having exhausted all possibilities, either $f(B_7') \cong A_7$ or $|G| > 7!$. We cannot have $G = A_7 $ by Artin, so $|G| \geq 7!$.

\end{proof}

Now for the $n=8$ case.

\begin{proof}
Lemmas \ref{big_PSL}, and \ref{small_PSL} deal with all homomorphisms $B_n' \to PSL(2,q)$, where $|PSL(2,q)| < 8!/2$ and $q \neq 16$. For the $q=16$ case, note that any homomorphism $B_7' \to PSL(2,16)$ must collapse $X_7'$, and therefore any map $B_8' \to PSL(2,q)$ must also collapse $X_7'$, meaning the map is trivial. Alternating groups $A_n$ for $n = 5,6,7$ are dealt with by a result from Lin \cite{Lin}.

We check the remaining groups, $PSL(3,3)$, $G_2(2)'$, $M_{11}$, and $PSL(3,4)$ with a computer. 

\end{proof}

\section{Proof of Theorem \ref{thmBound}}
\label{sectionOnBoundTheorem}

Theorem \ref{thmBound} is a refinement of the main theorem in \cite{finite_quotients}, which states that if $B_n\to G$ is non-cyclic and $n \geq 5$, then $|G| \geq 2^{\lfloor n/2 \rfloor -1}\lfloor n/2 \rfloor!$. We will present a brief summary of their proof so that we may explain certain modifications.

\vspace{.3cm}

Nominally, the proof given in \cite{finite_quotients} is a simple application of lemma \ref{TSS_criterion}---if $B_n \to G$ is non-cyclic, $G$ must contain a totally symmetric set of cardinality $\lfloor n/2 \rfloor$, and therefore $|G| \geq 2^{\lfloor n/2 \rfloor -1}\lfloor n/2 \rfloor!$. The real work in proving lemma \ref{TSS_criterion} is done by the following proposition:

\begin{lemma}
\label{prop_23}
Let $n \geq 1$, $G$ be a group, and $T \subset G$ be a totally symmetric subset with with $|T| = n$. Suppose that each element of $T$ has finite order, and let $p$ be the minimal integer such that $t_1^p = t_2^p$ for all $t_1,t_2 \in T$. Then $\langle T \rangle$ is a finite group whose order is greater than or equal to $p^{n-1}$.
\end{lemma}

The statement of this lemma in \cite{finite_quotients} does not define $p$. Instead its proof minimizes over all such values to obtain the bound $2^{n-1}$. Lemma \ref{TSS_criterion} now follows from lemma \ref{prop_23} by noting that $\langle T \rangle$ must lie in the kernel of the natural map $\text{Stab}(T) \to \text{Sym}(T)$ and therefore $|\text{Stab(T)}| \geq 2^{n-1}n!$.

\vspace{.3cm}

We are interested in the particular case where $G$ is the image of a map $f:B_n \to G$ and $T = f(X_n)$. Let $g_i = f(\s_i)$, so that $T = \{g_{2i-1}\}_{i=1}^{\lfloor n/2\rfloor}$ and $m$ be the order of the $g_i$'s. As in lemma \ref{prop_23}, we let $p$ be the minimal integer so that $g_1^p = g_3^p = \cdots = g_{2\lfloor n/d \rfloor-1}$.

When $m=2$, $G$ must be a non-cyclic quotient of $S_n$ and is therefore isomorphic to $S_n$. Because $n! >3^{\lfloor n/2 \rfloor - 1}\lfloor n/2 \rfloor !$, we will only consider cases where $m \geq 3$. To prove theorem \ref{thmBound}, it suffices to show that $p \geq 3$, after which the argument from \cite{finite_quotients} can be carried through, only replacing every $2^{\lfloor n/2 \rfloor -1}$ with a $3^{\lfloor n/2 \rfloor-1}$.

\begin{lemma}
For $n \geq 5$, if $G$ is the smallest non-cyclic quotient of $B_n$, then $p=m$.
\end{lemma}

\begin{proof}
We first show that $p\mid m$. If $m = pq + r$ with $0\leq r < p$, then $(g_1)^p = (g_{2i-1})^p$ for all $1\leq i \leq \lfloor n/2 \rfloor$, so $(g_1)^{pq}g_1^r = ((g_{2i-1})^p)^qg_{2i-1}^r = (g_1)^{pq}g_{2i-1}^r$, and therefore $g_1^r = g_{2i-1}^r$ for all such $i$. However, $p$ is the minimal positive integer with this property, so $r=0$.

We prove the contrapositive: assume $p < m$ and let $H = \langle g_1^p \rangle$. It suffices to show that $H$ is normal in $G$ and that $G/H$ is non-cyclic.

For normality, note that $g_1^p = g_3^p$ implies $g_1^p = g_4^p$ because $\sigma_3 \sigma_4$ conjugates $\sigma_3$ to $\sigma_4$ while fixing $\sigma_1$. Then $g_2 (g_1^p)^a g_2\inv = g_2 (g_4^p)^a g_2\inv = (g_4^p)^a = (g_1^p)^a$, and that $g_j (g_1^p)^a g_j\inv = (g_1^p)^a$ for any $j \neq 2$. 

Say $G/H$ is cyclic. Then, $g_1H = g_4H$ (being the images of $\s_1,\s_4$ under the composition of quotients $B_n \to G \to G/H$), so we have $g_1 = g_4(g_1^p)^a = g_4(g_4^p)^a$ for some $a$. Conjugating the above equality by $x = g_1g_2$, we obtain $g_2 = g_4(g_4^p)^a = g_4(g_1^p)^a = g_1$, so $g_1 = g_2$, contradicting the non-cyclicity of $G$.

\end{proof}

\section{Appendix}

Our computer code \cite{code}, which can be found at {\ttfamily https://github.com/Noah-Caplinger/Small- Quotients-of-Braids}, is split between three files: {\ttfamily TotallySymmetricSets.sage}, {\ttfamily Homomorphism CheckingAlgorithms.sage}, and {\ttfamily Results.sage}

\vspace{.3cm}

{\ttfamily TotallySymmetricSets.sage} contains all basic functions relating to totally symmetric sets---most notably a function which list all totally symmetric sets in in a group, and another which considers these totally symmetric sets up to conjugation, that is, returns a list containing a single totally symmetric set from each conjugacy class. For convenience, we use totally symmetric \textit{lists}---a totally symmetric set comes equipped with an ordering inherited with Sage's {\ttfamily conjugacy\_class()} method.

This file also contains permutation representations used for each group under consideration. Unless otherwise marked, these representations come from the Atlas of Finite Group Representations.

{\ttfamily HomomorphismCheckingAlgorithms.sage} Contains the functions used to check whether certain homomorphisms $B_n, B_n' \to G$ can exist. Fundamentally, these algorithms are just checking whether there are valid images in $G$ for the generators of $B_n$ and $B_n'$ which also obey the relations of their respective presentation. We use the standard presentation for $B_n$, and Lin's presentation of $B_n'$, which has generators

$$u = \s_2, \s_1\inv \qquad v = \s_1\s_2\s_1^{-2} \qquad w = \s_2\s_3\s_1\inv\s_2\inv \qquad  c_i = \s_{i+2}\s_1\inv \text{ for } 1 \leq i \leq n-3.$$ 

\noindent
A full list of relations can be found in \cite{Lin}. 

The main novelty of our approach is to pre-compute all totally symmetric sets in $G$, then use this to inform our otherwise exhaustive search for images of generators. For instance, when checking for non-trivial homomorphisms $B_7' \to G$, we know the totally symmetric sets $\{c_1,c_3\}$ and $\{c_2,c_4\}$ must be sent to two totally symmetric set of cardinality 2 in $G$. After computing all such totally symmetric sets directly, we can loop through all possible combinations of images for the totally symmetric sets instead of looping through all images of the elements $c_1,c_2,c_3,$ and $c_4$ directly. Furthermore, because any non-trivial map has non-trivial conjugates, we may consider only the totally symmetric sets of $G$ up to conjugation as images of $\{c_1,c_3\}$. This drastically reduces the search space.

Although we are not prepared to make any mathematical claims about the complexity of this algorithm, from a practical point of view it has made many computations feasible. In particular, the $n=8$ checks would be infeasible using a naive algorithm. 

The file {\ttfamily TotallySymmetricSets.sage} simply executes all homomorphism checks used in this paper. Run this file to reproduce these computations.

\vspace{1cm}

Next, we give a proof of lemma \ref{TSS_in_linear}.

\begin{proof}
Assume without loss of generality we are working over an algebraically closed field.

If any element of $S$ is diagonalizable, every element of $S$ is. In this case, we have commuting diagonalizable matrices, which can be simultaneously diagonalized. Then there is some matrix $X$ so that $XSX\inv$ is totally symmetric of diagonal matrices. Because they are conjugate, elements of $XSX\inv$ have the same eigenvalues. There are only $2$ distinct $2\times 2$ diagonal matrices with the same eigenvalues, so $|S| \leq 2$.

Let $S = \{A,B,C\}$ be a totally symmetric set of non-diagonalizable matrices. Then, there are some $v, \lambda$ so that $Av = \lambda v$. Then, 

$$A(Bv) = B(Av) = B \lambda v = \lambda Bv$$

So, $Bv$ is an eigenvalue of $A$. Because $S$ is not diagonalizable, the eigenspace of $\lambda$ (with respect to $A$) must be 1-dimensional, so $Bv$ is a scalar multiple of $v$. Then, $v$ is also an eigenvector of $B$ of the same eigenvalue since we assume $A$ is non-diagonalizable, and $A,B$ share eigenvalues.   

Choose some $w$ so that $\{v,w\}$ is a basis, and let $P$ be the change of basis matrix $\begin{pmatrix}
\vert & \vert \\
v & w \\
\vert & \vert 
\end{pmatrix}$

Then, $PSP\inv$ is a set of three non-diagonalizable upper-triangular matrices with eigenvalue $\lambda$:

$$PSP\inv = \left \{ \begin{pmatrix}
\lambda & a \\
0& \lambda 
\end{pmatrix},\begin{pmatrix}
\lambda & b \\
0& \lambda 
\end{pmatrix},\begin{pmatrix}
\lambda & c \\
0& \lambda 
\end{pmatrix} \right \} $$

Because $PSP\inv$ is totally symmetric, there is some matrix $\begin{pmatrix}
x & y \\
z & w
\end{pmatrix}$ which transposes $PAP\inv$ and $PBP\inv$ by conjugation, and fixes $PCP\inv$. That is,

$$\begin{pmatrix}
x & y \\
z & w
\end{pmatrix} \begin{pmatrix}
\lambda & a \\
0& \lambda 
\end{pmatrix} \begin{pmatrix}
x & y \\
z & w
\end{pmatrix}\inv = \begin{pmatrix}
\lambda & b \\
0& \lambda 
\end{pmatrix}$$

By inspection, $\begin{pmatrix}
x \\
z
\end{pmatrix}$ is an eigenvalue of the LHS, but the only eigenvalue of the RHS are of the form $\begin{pmatrix}
s \\ 0
\end{pmatrix}$, so we must have $z = 0$. A computation shows 

$$\begin{pmatrix}
x & y \\
0 & w
\end{pmatrix} \begin{pmatrix}
\lambda & c \\
0& \lambda 
\end{pmatrix} \begin{pmatrix}
x & y \\
0 & w
\end{pmatrix}\inv = \begin{pmatrix}
\lambda & \frac{x}{w}c \\
0& \lambda 
\end{pmatrix}.$$

Because this matrix fixes $PCP\inv$, we know that $\frac{x}{w}c = c$, so $\frac{x}{w} = 1$. 

The same computation above shows 

$$\begin{pmatrix}
x & y \\
0 & w
\end{pmatrix} \begin{pmatrix}
\lambda & a \\
0& \lambda 
\end{pmatrix} \begin{pmatrix}
x & y \\
0 & w
\end{pmatrix}\inv = \begin{pmatrix}
\lambda & \frac{x}{w}a \\
0& \lambda 
\end{pmatrix} = \begin{pmatrix}
\lambda & a \\
0& \lambda 
\end{pmatrix}.$$

Thus $\begin{pmatrix}
x & y \\
z & w
\end{pmatrix}$ does not transpose $PAP\inv$ and $PBP\inv$ as we assumed.

\end{proof}

\bibliographystyle{plain}
\bibliography{bib}

\end{document}